\documentclass[12pt]{article}
\usepackage{fullpage,amsthm,amssymb,amsmath}
\usepackage{graphicx,algorithmic,algorithm,amsfonts, verbatim}
\usepackage{enumerate}
\usepackage{color}
\usepackage{blindtext,caption,subcaption,subfig}
\usepackage[pagewise]{lineno}

\newtheorem{theorem}{Theorem}[section]
\newtheorem{lemma}[theorem]{Lemma}
\newtheorem{proposition}[theorem]{Proposition}
\newtheorem{corollary}[theorem]{Corollary}

\newtheorem*{GRID}{Theorem \ref{thm:gridthm}}
\newtheorem*{OBS}{Theorem \ref{thm:2}}

\newcommand{\rNum}[1]{\lowercase\expandafter{\romannumeral #1\relax}}

\newcommand\abs[1]{\lvert #1\rvert}
\newcommand{\mm}{\operatorname{mm}}
\newcommand{\mmw}{\operatorname{mmw}}
\newcommand{\MMW}{mm-width }
\newcommand{\rw}{\operatorname{rw}}
\newcommand{\brw}{\operatorname{brw}}
\newcommand{\tw}{\operatorname{tw}}

\begin{document}
\title{Characterizing graphs of maximum matching width at most $2$}
\author{Jisu Jeong, Seongmin Ok, Geewon Suh}
\date\today
\maketitle
\begin{abstract}
The maximum matching width is a width-parameter that is defined on a branch-decomposition over the vertex set of a graph.
The size of a maximum matching in the bipartite graph is used as a cut-function.
In this paper, we characterize the graphs of maximum matching width at most $2$ using the minor obstruction set. 
Also, we compute the exact value of the maximum matching width of a grid.
\end{abstract}

\section{Introduction}
Treewidth and branchwidth are well-known width-parameters of graphs used in structural graph theory and theoretical computer science. Based on Courcelle's theorem~\cite{Courcelle1990}, which states that every property on graphs definable in monadic second-order logic can be decided in linear time on a class of graphs with bounded treewidth, many NP-hard problems have been shown to be solvable in polynomial time by the dynamic programming when the input has bound treewidth or branchwidth.

Vatshelle~\cite{Vatshelle2012} introduced a new graph width-parameter, called the maximum matching width (\MMW in short), 
that uses the size of a maximum matching as a cut-function in its branch-decomposition over the vertex set of a graph. 
Maximum matching width is related to treewidth and branchwidth as shown by the inequality $\mmw(G)\le \max(\brw(G),1)\le \tw(G)+1 \le 3\mmw(G)$ for every graph $G$~\cite{Vatshelle2012} where $\mmw(G)$, $\tw(G)$, and $\brw(G)$ is the maximum matching width, the treewidth, and the branchwidth of $G$ respectively.
This implies that bounding the treewidth or branchwidth is qualitatively equivalent to bounding the maximum matching width.
Maximum matching width gives a more efficient algorithm for some problems. For a given branch-decomposition of a graph $G$ of maximum matching width $k$, we can solve Minimum Dominating Set Problem in time $O^* (8^k)$~\cite{JST15}, which gives a better runtime than $O^*(3^{\tw(G)})$-time algorithm in~\cite{RBR2009} when $\tw(G)>(\log_3 8) k$. Remark that Minimum Dominating Set Problem can not be solved in time $O^*((3-\varepsilon)^{\tw(G)})$ for every $\varepsilon>0$ unless the Strong Exponential Time Hypothesis fails~\cite{DDS2011}.

Robertson-Seymour theorem~\cite{RS04} states that every minor-closed class of graphs has a finite minor obstruction set.
In the other words, a graph $G$ is in the class if and only if $G$ has no minor isomorphic to a graph in the obstruction set. Much work has been done to identify the minor obstruction set for various graph classes, especially for graphs of bounded width-parameters~\cite{Bodlaender2015,Dvorak2012,Kinnersley1994}. 

Let $K_n$, $C_n$, and $P_n$ be the complete graph, the cycle graph, and the path graph on $n$ vertices, respectively.
The graph $K_3$ and $K_4$ is the unique minor obstruction for the graphs of treewidth at most $1$ and $2$~\cite{WC1983}, respectively. 
The minor obstruction set for a class of graphs having treewidth at most $3$ is $\{K_5,K_{2,2,2},K_2 \times C_5,M_8\}$ where $K_2 \times C_5$ is the Cartesian product of $K_2$ and $C_5$, and $M_8$ is the Wagner graph, also called the M\"{o}bius ladder with eight vertices~\cite{Arnborg1990,ST1990}.

Robertson and Seymour~\cite{RS1991} gave a characterization for the classes of graphs of branchwidth at most $1$ and at most $2$.
The graphs $K_3$ and $P_4$ are forbidden minors for the graphs of branchwidth at most $1$.
For the class of graphs of branchwidth at most $2$, its minor obstructions is the same as treewidth, which is $K_4$.
The graphs of branchwidth at most $3$ have four minor obstructions; $\{K_5,K_{2,2,2},K_2 \times C_4, M_8\}$~\cite{BT1999}.

One of the main results of this paper is to find the minor obstruction set for the class of graphs of \MMW at most $2$.
Note that the class of graphs with bounded \MMW is closed under taking minor, as shown in Corollary \ref{coro:closed}.
\begin{OBS} 
Let $\mathcal{O}=\mathcal{O}_3\cup\mathcal{O}_4\cup\mathcal{O}_5\cup\mathcal{O}_6$ be the set of $42$ graphs in Figures~\ref{fig:3conn},\ref{fig:o4},\ref{fig:o5},\ref{fig:o6}. A graph $G$ has \MMW at most 2 if and only if $G$ has no minor isomorphic to a graph in $\mathcal{O}$.
\end{OBS} 

The exact value of some width-parameters for grid graphs are well known. For an integer $k\ge 1$, the branchwidth and treewidth of the $k\times k$-grid are $k$~\cite{RS1991,ST1993}, and the rank-width of the $k\times k$-grid is $k-1$~\cite{Jelinek2010}. From the inequality $\rw(G)\le\mmw(G)\le\max(\brw(G),1)$~\cite{Vatshelle2012}, the \MMW of the $k\times k$-grid is either $k-1$ or $k$. Our second result is that the latter is the right answer when $k \geq 2$.

\begin{GRID}\label{thm:grid}
The $k\times k$-grid has \MMW $k$ for $k\ge 2$.
\end{GRID}

Section~\ref{sec:prelim} lists some of the definitions, including a \emph{tangle}, and provides preliminaries for the maximum matching width.
In Section~\ref{sec:mmw2} we identify the minor obstruction set for graphs with \MMW at most $2$. 
Section~\ref{sec:grid} is for the precise \MMW of the square grids.

\section{Preliminaries}\label{sec:prelim}

Every graph $G=(V,E)$ in this paper is finite and simple. For a set $X \subseteq V(G) \cup E(G)$, we write $G\setminus X$ to denote the graph obtained from $G$ by deleting all vertices and edges in $X$. If $X \subseteq E(G)$, we write $G/X$ to denote the graph obtained from $G$ by contracting the edges in $X$. If $X = \{x\}$, then we write $G\setminus x$ and $G/x$ instead of $G\setminus X$ and $G/X$, respectively.
If a subgraph $G'$ of $G$ with $V(G') = X$ contains all the edges of $G$ whose both ends are in $X$, then we call $G'$ \emph{induced by} $X$ and write $G':=G[X]$.
For a graph $G$ and disjoint subsets $X,Y\subseteq V(G)$, let $E_G[X,Y]$ (or $E[X,Y]$) denote the set of all edges $e=uv$ where $u$ is in $X$ and $v$ is in $Y$, and let $G[X,Y]=G(X\cup Y,E[X,Y])$. 
A graph $G$ is \emph{$k$-connected} if $\abs{V(G)}\ge k$ and $G\setminus X$ is connected for every $X\subset V(G)$ with $\abs{X}<k$. A \emph{bridge} is an edge $e$ such that $G\setminus e$ has more components than $G$. 
A \emph{block} is 
either a bridge as a subgraph or 
a maximal $2$-connected subgraph. 

We say that a tree is \emph{subcubic} if all vertices have degree $1$ or $3$.
A \emph{branch-decomposition} of a finite set $X$ is a pair $(T,\mathcal{L})$ of a subcubic tree $T$ together with a bijection $\mathcal{L}$ from the leaves of $T$ to $X$.
Note that an edge $ab$ of $T$ partitions the leaves of $T$ into two parts, say $A$ and $B$. 
We say an edge $e$ \emph{induces} the partition $(A,B)$.
A function $f:2^X\rightarrow \mathbb{Z}$ is \emph{symmetric} if $f(A)=f(X\setminus A)$ for all $A\subseteq X$, and
the function $f$ is \emph{submodular} if $f(A)+f(B)\ge f(A\cup B)+f(A\cap B)$ for all $A,B\subseteq X$.  
For each edge $e$ of $T$, and a symmetric, submodular function $f$, the \emph{$f$-value} of $e$ is equal to $f(A)=f(B)$ where $(A,B)$ is the partition induced by $e$.
The \emph{$f$-width} of a branch-decomposition $(T,\mathcal{L})$ is the maximum $f$-value of an edge of $T$, and 
the \emph{$f$-width} of $X$ is the minimum value of the $f$-width over all possible branch-decompositions of $X$. 
This notion of $f$-width provides a link between several width parameters.

For $A\subseteq E(G)$, let $br:2^{E(G)} \rightarrow \mathbb{Z}$ be the function so that $br(A)$ is the number of vertices that are incident to both an edge in $A$ and an edge in $E(G)\setminus A$. The \emph{branchwidth} of $G$, denoted by \emph{$\brw(G)$}, is the $br$-width over $E(G)$.

For $A\subseteq V(G)$, let $r:2^{V(G)} \rightarrow \mathbb{Z}$ be the function such that $r(A)$ is the rank of the adjacency matrix between $A$ and $V(G)\setminus A$ over $\mathbb{F}_2$. The \emph{rank-width} of $G$, denoted \emph{$\rw(G)$}, is the $r$-width over $V(G)$.

Let $\mm_G:2^{V(G)} \rightarrow \mathbb{Z}$ be the function such that $\mm_G(A)$ is the size of a maximum matching in $G[A,V(G)\setminus A]$. 
Note that the function $\mm_G$ is symmetric and submodular~\cite{ST2014}.
We use $\mm$ instead of $\mm_G$ if the host graph $G$ is clear from the context.
The \emph{maximum matching width} of $G$, denoted $\mmw(G)$, is the $\mm$-width over $V(G)$. 

A graph $H$ is a \emph{minor} of a graph $G$ if $H$ can be constructed from $G$ by deleting edges, deleting vertices, and contracting edges. We call a graph $G$ \emph{minor-minimal} with respect to a property $\mathcal{P}$ if $G$ has $\mathcal{P}$ but no proper minor of $G$ has $\mathcal{P}$. A graph $G$ is a \emph{forbidden minor} of a graph class $\mathcal{C}$ when $H \notin \mathcal{C}$ if $H$ has a minor isomorphic to $G$. Robertson and Seymour~\cite{RS04} state that the collection of minor-minimal graphs outside a minor-closed graph class is finite. The collection is called the \emph{minor obstruction set}. 

A graph is \emph{chordal} if every cycle $C$ of length at least $4$ has an edge, which is not contained in $E(C)$, connecting two vertices of $C$. A \emph{chordalization} of a graph $G$ is a chordal graph $H$ such that $V(H)=V(G)$ and $E(G)\subseteq E(H)$. 
An \emph{intersection graph} $G$ over a family $\{A_i\}$ of sets is the graph with $V(G)=\{A_i\}$ and $E(G)=\{A_i A_j : A_i\cap A_j \not = \emptyset\}$. 
Remark that a graph is chordal if and only if it is the intersection graph of the edge sets of  subtrees of a tree~\cite{Gavril1974}. 

\subsection{Maximum matching width}

Jeong, S\ae ther, and Telle~\cite{JST15} gave a new characterization of graphs of \MMW at most $k$ as an intersection graph by the following theorem. A tree is called \emph{nontrivial} if it has at least one edge.

\begin{theorem}[\cite{JST15}]\label{thm:newchar} 
The maximum matching width of a graph $G$ is at most $k$ if and only if there exist a subcubic tree $T$ and a set $\{T_x\}_{x\in V(G)}$ of nontrivial subtrees of $T$ such that 
\begin{itemize}
\item[(1)] if $uv\in E(G)$, then the subtrees $T_u$ and $T_v$ have at least one vertex of $T$ in common, 
\item[(2)] for each edge $e$ of $T$ there are at most $k$ subtrees in $\{T_x\}_{x\in V(G)}$ containing $e$.
\end{itemize}
\end{theorem}

A \emph{tree-representation} of $G$ having width at most $k$ is a pair $(T,\{T_x\}_{x\in V(G)})$ where $T$ is a subcubic tree and a set $\{T_x\}_{x\in V(G)}$ of nontrivial subtrees satisfying the properties (1) and (2). Theorem~\ref{thm:newchar} says that a graph $G$ has a tree-representation of width at most $k$ if and only if $\mmw(G)\le k$.

For a tree-representation $(T,\{T_x\})_{x\in V(G)}$ of $G$, the intersection graph $G_T$ of the family $\{T_x\}_{x\in V(G)}$ is chordal and $G$ is a subgraph of $G_T$. 
Since $G$ and $G_T$ have the same tree-representation $(T,\{T_x\})_{x\in V(G)}$, 
every graph has a chordalization with the same mm-width.


It is easy to check that, for a graph $G$ and its vertex or edge $x$, 
\[
\mmw(G\setminus x)\le \mmw(G).
\]

\begin{lemma}
Let $G$ be a graph. For every edge $uv$ of $G$, $\mmw(G/uv) \leq \mmw(G)$. \end{lemma}
\begin{proof}
Let $(T,\{T_x\}_{x\in V(G)})$ be a tree-representation of $G$ having width $\mmw(G)$.
Let $T_{uv}$ be the subtree of $T$ with vertex set $V(T_u)\cup V(T_v)$ and edge set $E(T_u)\cup E(T_v)$.
Then $(T,\{T_x\}_{x\in V(G)\setminus\{u,v\}}\cup\{T_{uv}\})$ is a tree-representation of $G/uv$ having  width at most $\mmw(G)$.
By Theorem \ref{thm:newchar}, $\mmw(G/uv) \leq \mmw(G)$.
\end{proof}

\begin{corollary}\label{coro:closed}
Let $k$ be an integer. The set $M_k = \{ G: \mmw(G)\le k \}$ is closed under the minor operations.
\end{corollary}

By Corollary~\ref{coro:closed} and Robertson-Seymour theorem~\cite{RS04}, $M_k$ has a finite minor obstruction set for each $k$. We can easily find the minor obstruction set when $k=1$.

\begin{proposition}[\cite{STpersonal}]
A graph $G$ has \MMW at most $1$ if and only if $G$ does not contain $C_4$ as a minor.
\end{proposition}
\begin{proof}
Suppose that $G$ contains $C_4$ as a minor. We can find four vertices $v_1,v_2,v_3,v_4$ of $G$ and four paths $P_{12}, P_{23}, P_{34}, P_{41}$ in $G$ such that each path $P_{ij}$ is a path from $v_i$ to $v_j$ and the four paths are pairwise internally vertex-disjoint.
For every branch-decomposition $(T,\mathcal{L})$ of $G$, there exists an edge $e$ in $T$ that  induces a partition $(A,B)$ of $V(G)$ such that two vertices from $v_1,v_2,v_3,v_4$ are in $A$ and the other two are in $B$. Thus, there exist two vertex-disjoint paths from $A$ to $B$. 
This implies that the $\mm_G$-value of $e$ is at least $2$, and therefore $G$ has \MMW at least $2$.

Now let us suppose that $G$ does not contain $C_4$ as a minor. It is easy to see that every block of $G$ is either $K_2$ or $C_3$.
The \MMW of $G$ is the maximum value among the mm-widths of blocks of $G$.
Since both $K_2$ and $C_3$ have \MMW $1$, $G$ has \MMW at most $1$.
\end{proof}

\subsection{Tangle}\label{sec:tangle}

Before proving our main theorems, we shall introduce the notion of \emph{tangle}, which is useful in investigating the lower bounds of width-parameters.

Let $f$ be an integer-valued symmetric submodular function on the subsets of a finite set $X$.
An \emph{$f$-tangle} of order $k+1$ is a collection $\mathcal{T}$ of subsets of $X$ satisfying that
\begin{itemize}
\item[(T1)] for all $S\subseteq X$, if $f(S)\le k$, then one of $S$ and $X \setminus S$ is in $\mathcal{T}$,
\item[(T2)] if $S_1,S_2,S_3\in \mathcal{T}$, then $S_1\cup S_2\cup S_3 \neq X$,
\item[(T3)] for each $x\in X$, $X \setminus \{x\} \notin \mathcal{T}$.
\end{itemize}

Robertson and Seymour~\cite{RS1991} proved the following theorem. 
We use it in both Section~\ref{sec:mmw2} and Section~\ref{sec:grid}.

\begin{theorem}[\cite{RS1991}]\label{thm:tangle}
Let $f$ be an integer-valued symmetric submodular function on subsets of a finite set $X$.
The $f$-width of $X$ is larger than $k$ if and only if
there exists an $f$-tangle of order $k+1$.
\end{theorem}

The \emph{$k\times k$-grid}, denoted by $G_k$, is the graph with vertex set $V(G_k)=\{(i,j):1\le i,j \le k\}$ and edge set $E(G_k)=\{(i,j)(i',j'):\abs{i-i'}+\abs{j-j'}=1 \}$. Using Theorem~\ref{thm:tangle}, we show that the $3\times 3$-grid has \MMW $3$, as an example. 

\begin{lemma}\label{lem:tangle3}
The $3\times 3$-grid $G_3$ has an $\mm$-tangle of order $3$.
\end{lemma}
\begin{proof}
Let us consider $G_3$ to be a part of an integer grid in the real plane and let $\{(i,j): 1 \leq i,j \leq 3 \}$ be the vertex set of $G_3$.
Let $A$ be a set of all subsets of $V(G_3)$ with size at most $2$.
Let $B=\{\{(1,1),(1,2),(2,1)\},$ $\{(1,2),(1,3),(2,3)\},$ $\{(2,3),(3,2),(3,3)\},$ $\{(2,1),(3,1),(3,2)\}\}$.
We claim that $A\cup B$ is an $\mm$-tangle of order $3$.
It is trivial that (T3) holds. If $S_1 \cup S_2 \cup S_3 = V(G_3)$, then the sets $S_1, S_2, S_3$ must be in $B$. However, no set in $B$ has $(2,2)$ and thus (T2) follows. Now we check (T1). Note that for every subset $S\subseteq V$ with $\abs{S} = 4$, we have $\mm(S) \geq 3$.
Since $A$ contains all subsets of size at most $2$,
we need to consider subsets of $V(G_3)$ of size $3$.
The elements in $B$ are the only subsets of size $3$ having $\mm(S)\le 2$.
Hence (T1) holds too and $A \cup B$ is a $\mm$-tangle of order $3$.
\end{proof}

By Lemma~\ref{lem:tangle3} and Theorem~\ref{thm:tangle}, the $3 \times 3$-grid has \MMW at least $3$.
It is easy to see that the $3\times 3$-grid has \MMW at most $3$ since it has $9$ vertices and $K_9$ has a tree-representation of width $3$.
Thus the $3\times 3$-grid has \MMW $3$.
In this paper, we use a similar argument to verify that 
the graphs in the minor obstruction set for \MMW at most $2$ 
has \MMW $3$.
Note that the $3\times 3$-grid is also in the minor obstruction set for the graphs of \MMW at most $2$. See Figure~\ref{fig:o52}.

\section{Minor obstruction set for maximum matching width at most $2$}\label{sec:mmw2}

Note that if $G$ is not $2$-connected, then $\mmw(G)$ is the maximum of $\mmw(H)$ where $H$ is a maximal $2$-connected subgraph of $G$.
Thus the graphs in the minor obstruction set are $2$-connected.

In Section 3.1 we identify the 3-connected graphs that are minor-minimal with respect to mm-width $\geq 3$. And then we consider the minor-obstructions with 2-cuts in Section 3.2. We shall show that each 2-cut separates the graph into at most three components, where all but one components are small (a full characterization is given after Lemma 3.7). And we show that the obstructions are obtained from a 3-connected graph with $\leq 6$ vertices by replacing some edges with small components mentioned above. What remains is to check all the candidates.

\subsection{3-connected graphs}

In this subsection, we give five $3$-connected graphs that have \MMW $3$ and whose proper minors have \MMW $2$.

Let $T$ be a subcubic tree. We can always find an edge of $T$ whose removal divides the set of leaves into two subsets, each having at least $1/3$ of all the leaves. 
Let $e=uv$ be an edge that induces a partition $(A,B)$ of the leaves where $u$ is on the side of $A$. Suppose that $A$ contains more than $2/3$ of the leaves. Then $u$ has degree $3$ and the other two edges at $u$ induce leaf partitions, namely $(A_1, B_1)$ and $(A_2, B_2)$ where we assume $u$ to be on the side of $A_1$ and $A_2$ respectively. We choose the edge, say $e'$, with larger $|A_i|$. 
If both $A_1$ and $A_2$ contain at most $2/3$ of the leaves then $e'$ will be the edge we are after. Otherwise, we have a partition with smaller difference $|A_i| - |B_i|$ than $|A| - |B|$ and we iterate until we find a working edge.

Therefore, a subcubic tree with at least $7$ leaves has an edge dividing the leaves into two sets such that both have size at least $3$.

\begin{lemma} \label{lem:3conn mmw}
If a graph $G$ is $3$-connected and $G$ has at least $7$ vertices, then $\mmw(G)\ge 3$.
\end{lemma}
\begin{proof}
By the argument above, for every branch decomposition $(T,\mathcal{L})$ of $V(G)$, we can find an edge $e$ in $T$ inducing a partition $(A,B)$ with $\abs{A},\abs{B} \geq 3$.
Since $G$ is $3$-connected, by Menger's theorem, $G$ has three vertex-disjoint paths between $A$ and $B$.
These paths give a matching of size $3$ in $G[A,B]$, which means that the $\mm_G$-value of $e$ is at least $3$.
Thus, every branch-decomposition of $V(G)$ has $\mm_G$-width at least $3$.
\end{proof}

It is easy to find a tree-representation of $K_{3n}$ with width $n$. 
In particular, $K_6$ has \MMW $2$ and hence every graph on $6$ vertices has \MMW at most $2$. 
In other words, the forbidden minors for \MMW at most $2$ have at least $7$ vertices.
We use the Tutte's wheel theorem stated below. In the following statement we assume pairwise parallel edges occuring from contractions are all removed but one to keep the graph simple.

\begin{theorem}[The Tutte's wheel theorem \cite{Tutte1961}]
If a graph $G$ is $3$-connected, then $G$ has an edge $e$ such that either $G/e$ or $G\setminus e$ is $3$-connected unless $G=K_4$.
\end{theorem}

\begin{figure}[h!]
\center
\includegraphics[scale=0.6]{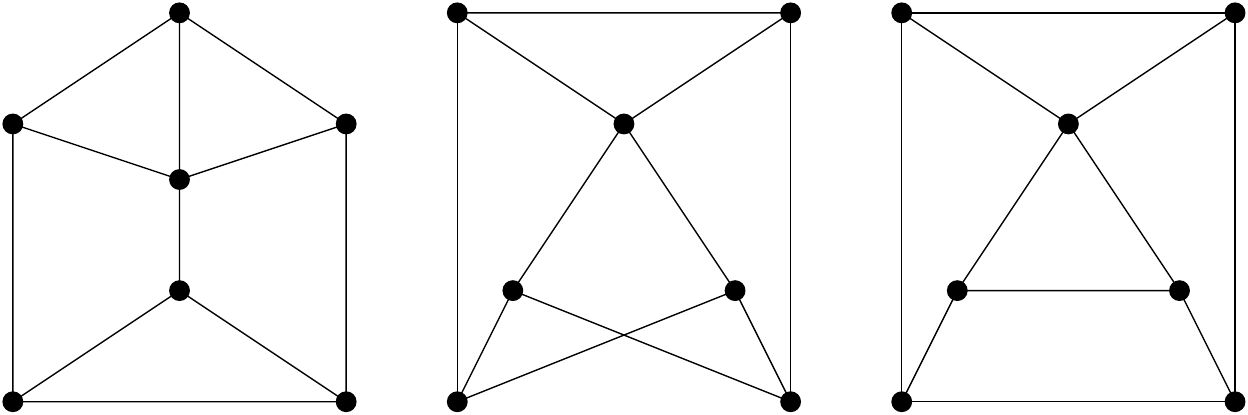}
\qquad
\includegraphics[scale=0.6]{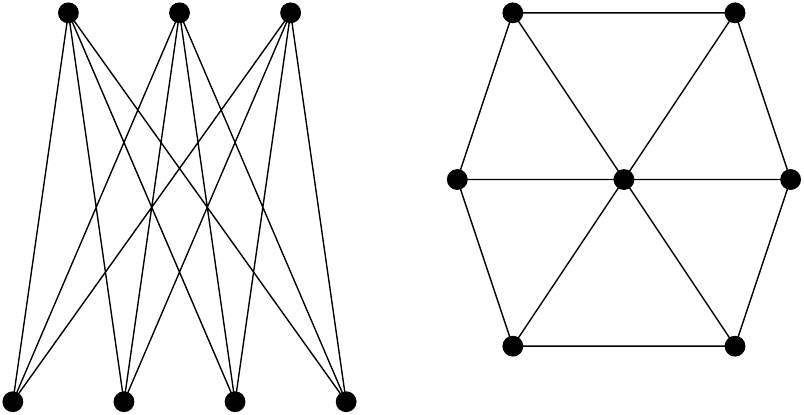}
\caption{The minor-minimal $3$-connected graphs on $7$ vertices}\label{fig:3conn}
\end{figure}

\begin{lemma}\label{lem:O3}
Let $\mathcal{O}_3$ be the set of the five graphs in Figure~\ref{fig:3conn}.
A $3$-connected graph is minor-minimal with respect to maximum matching width at least $3$ if and only if 
it is in $\mathcal{O}_3$.
\end{lemma}

\begin{proof}
By the Tutte's wheel theorem, a $3$-connected graph with at least $8$ vertices has a proper $3$-connected minor with at least $7$ vertices, which has \MMW at least $3$ by Lemma \ref{lem:3conn mmw}. Thus a minor-minimal $3$-connected graph has precisely $7$ vertices. By~\cite{read1998atlas}, the five graphs in Figure~\ref{fig:3conn} are precisely the edge-minimal $3$-connected graphs on $7$ vertices, and hence it is enough to show that the proper minors of these graphs all have \MMW at most $2$.

Observe that all edges of a graph in $\mathcal{O}_3$ are incident with a vertex of degree $3$. Thus by taking out the edge we have a graph on $7$ vertices with at least one vertex of degree $2$, say $v$. Starting from a tree-representation of $G\setminus v$ with width $2$, by rearranging the leaves if needed, we can easily add a vertex $v$ of degree $2$ without increasing mm-width, so such a graph must have \MMW $2$.
\end{proof}

\subsection{2-connected graphs}
Now we find $2$-connected minor-minimal graphs with respect to \MMW $3$ that are not 3-connected.
Let $\mathcal{O}_2$ be the set of all graphs $G$ such that $G$ is not $3$-connected and $G$ is minor-minimal with respect to \MMW at least $3$. Note that the graphs in $\mathcal{O}_2$ are 2-connected.

Let $G$ be a graph and let $a,b \in V(G)$. 
We say that a tree-representation of $G$ is \emph{good} 
if there exist two vertices $a$ and $b$ such that the subtrees for $a$ and $b$ share an edge and the width of the tree-representation is $2$. A pair $(G, \{a,b\})$ is \emph{good} if it has a corresponding good tree-representation with vertices $a$ and $b$, and \emph{bad} if none exists.

\begin{lemma} \label{lem:bad+p4}
Let $G$ be a graph and let $a,b \in V(G)$. Let $H$ be the graph obtained from $G$ by adding two new vertices, say $c$ and $d$, and edges $ac, cd$ and $db$, followed by removing the edge $ab$ if $ab \in E(G)$. If $(G,\{a,b\})$ is bad, then $\mmw(H) \geq 3$.
\end{lemma}
\begin{proof}
We prove by contradiction. Suppose $\mmw(H) \leq 2$, that is, $H$ has a tree-representation $\mathcal{T} = (T,\{T_v\}_{v\in V(H)})$ of width at most $2$. We shall use $\mathcal{T}$ to find a good tree-representation of $G$ with $a$ and $b$, yielding a contradiction.

From $\mathcal{T}$ we may obtain three tree-representations of $G$ with width at most 2 by replacing the subtree for $a$ and $b$ respectively with (1) $T_a \cup T_c \cup T_d$ and $T_b$, (2) $T_a \cup T_c$ and $T_d \cup T_b$, and (3) $T_a$ and $T_c \cup T_d \cup T_b$.
Since $(G,\{a,b\})$ is bad, for all three choices the subtrees for $a$ and $b$ intersect at precisely one vertex in the new tree-representations. 
Therefore, $E(T_a) \cap E(T_b) = \emptyset$ and $T$ has two distinct vertices $v_1$ and $v_2$ such that $T_a \cap T_c = \{v_1\}$ and $T_d \cap T_b = \{v_2\}$.

Let $e=v_1u$ be the first edge in the unique path $P$ in $T$ from $v_1$ to $v_2$. Because of the path $acdb$ in $H$, the first few consecutive edges of $P$ (possibly zero) are in $T_c$ and the others are in $T_d$. We assume that $T_c$ contains $e$. The following manipulation can be done likewise when $T_d$ contains $e$.

\begin{figure}[h!]
\center
\includegraphics[scale=0.8]{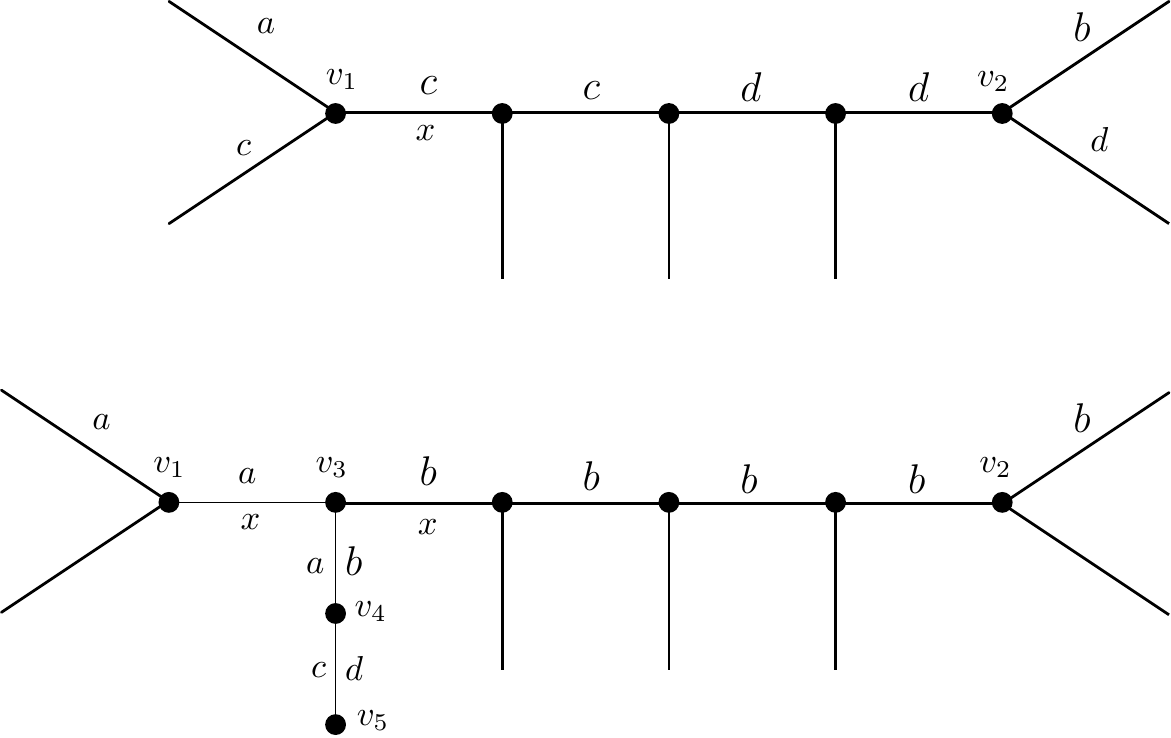}
\caption{Another tree-representation of the same graph where $T_a$ and $T_b$ share an edge}\label{fig:2conntrrep}
\end{figure}

Let $x \neq c$ be a vertex of $G$ such that $T_x$ contains $e$. If there is no such $x$ we ignore $x$ in the following. Let $T'$ be the tree obtained from $T$ by subdividing $e$, and adding a path $v_3 v_4 v_5$ of length $2$ at the new vertex $v_3$ obtained from the subdivision; see Figure \ref{fig:2conntrrep}. Let $\{T'_v\}_{v \in V(H)}$ be a collection of subtrees of $T'$ such that
\begin{itemize}
\setlength{\itemsep}{0pt}
\setlength{\parsep}{0pt}
\setlength{\parskip}{0pt}
\item $T'_c$ and $T'_d$ have only one edge $v_4 v_5$,
\item $T'_a$ is obtained from $T_a$ by adding the edges $v_1 v_3$ and $v_3 v_4$,
\item $T'_b$ is obtained from $T_b$ by adding $v_3 v_4$ and the edges on the path from $v_3$ to $v_2$, and
\item $T'_v = T_v$ for all other $v \in V(H)$.
\end{itemize}

Note that for each pair of vertices $u,v$ in $V(H) \setminus \{c,d\}$, if $T_u \cap T_v \neq \emptyset$ then $T'_u \cap T'_v \neq \emptyset$. Since $c$ and $d$ are adjacent to only $a$ and $b$ in $H$, the pair $\mathcal{T'} = (T', \{T'_v\}_{v \in V(H)})$ is a tree-representation of $H$ of width 2. Hence, by removing $v_5$, we obtain a good tree-representation of $G$ with $a$, $b$ having width $2$, a contradiction.
\end{proof}

\begin{lemma} \label{lem:bad+p3}
Let $G$ be a graph and let $c$ be a vertex of $G$ with precisely two neighbors $a$ and $b$. If $ab \in E(G)$ and $\mmw(G) \geq 3$, then $\mmw(G \setminus ab) \geq 3$.
\end{lemma}
\begin{proof}
We prove by contradiction. Suppose $H = G \setminus ab$ has a tree-representation $\mathcal{T} = (T, \{T_v\}_{v \in V(H)})$ of width at most $2$. Since $\mmw(G) \geq 3$ the subtrees $T_a$ and $T_b$ are vertex-disjoint. Let $v_1 \in V(T_a)$ and $v_2 \in V(T_b)$ be the vertices of $T$ such that the unique path $P$ in $T$ from $v_1$ to $v_2$ have no edge in neither $T_a$ nor $T_b$. As $c$ is a common neighbor of $a$ and $b$, every edge of $P$ is in $T_c$. Now we do the same as in the proof of Lemma \ref{lem:bad+p4} and Figure \ref{fig:2conntrrep}, except that here we set $d=c$. The resulting tree-representation of $G$ has width at most $2$, a contradiction.
\end{proof}

A \emph{$2$-cut} in $G$ is an inclusion-wise minimal subset $S\subset V(G)$ such that $\abs{S}=2$ and $G\setminus S$ is disconnected. 
Given a graph $G$ and its 2-cut $\{a,b\}$ with a component $S$ of $G \setminus \{a,b\}$, we denote by $\tilde S$ the induced subgraph $G[V(S) \cup \{a,b\}]$. As $\{a,b\}$ is the unique 2-cut having $S$ as a component, we may say simply $\tilde S$ is \emph{good} or \emph{bad}.

\begin{lemma}\label{lem:2cut}
Let $G$ be a graph in $\mathcal{O}_2$.
If a $2$-cut $\{a,b\}$ separates $G$ into two components $A$ and $B$, 
then $ab \notin E(G)$ and one of $\tilde A$ or $\tilde B$ is isomorphic to either $P_3$ or $P_4$.
\end{lemma}
\begin{proof}
We start with showing that one of $\tilde A$ and $\tilde B$ is bad. Suppose for contradiction that both are good. From their good representations, say $(T^A, \{T_x\}_{x \in V(A)})$ and $(T^B, \{T_y\}_{y \in V(B)})$, we can construct a tree-representation of $G$ of width $2$ as follows. We choose an edge from each of $T^A$ and $T^B$ shared by $T_a$ and $T_b$, and then subdivide those two edges and connect the new vertices by an edge; see Figure~\ref{fig:2conntr}. The new subtrees $T_a'$ and $T_b'$ will be clear from Figure~\ref{fig:2conntr}. It is easy to see that the resulting tree-representation has width 2.

\begin{figure}[h!]
\center
\includegraphics[scale=0.8]{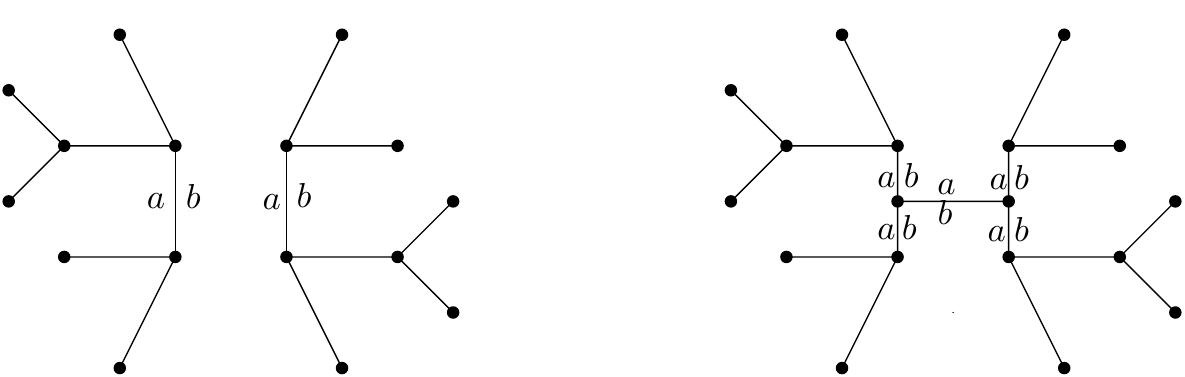}
\caption{New tree-representation from two tree-representations}\label{fig:2conntr}
\end{figure}
Now we assume $\tilde B$ is bad. Suppose $|A| \geq 2$. If $\tilde A$ is not a path between $a$ and $b$ of length 3, then by Lemma \ref{lem:bad+p4}, $G$ has a proper minor of \MMW $3$, a contradiction. Hence if $|A| \geq 2$, then $\tilde A$ is isomorphic to $P_4$. 
Suppose $A = \{c\}$. 
Since $G$ is $2$-connected,
$c$ is adjacent to both $a$ and $b$ and by Lemma \ref{lem:bad+p3}, $\tilde A$ is the path $acb$.
\end{proof}

To consider the 2-cuts with more than two components, we use the following lemma. The proof of Lemma~\ref{lem:bad+p4} can be modified to prove the following.

\begin{lemma} \label{lem:bad+c4}
Let $G$ be a graph and let $a,b \in V(G)$. Let $H$ be the graph obtained from $G$ by adding two new vertices, say $c$ and $d$, and edges $ac$, $bc$, $ad$ and $bd$, followed by removing the edge $ab$ if $ab \in E(G)$. If $(G,\{a,b\})$ is bad, then $\mmw(H) \geq 3$.
\end{lemma}

Suppose that a $2$-cut $\{a,b\}$ separates $G$ into at least three components, namely $D_1, D_2, \ldots, D_k$. Since we can combine the arbitrary number of good tree-representations as in Figure~\ref{fig:2conntr} while preserving goodness, one of the $\tilde D_i$'s, say $\tilde D_1$, is bad. Because of the previous paragraph, we have $k \leq 3$ and one of the following holds:
\begin{enumerate}
\setlength{\itemsep}{0pt}
\setlength{\parsep}{0pt}
\setlength{\parskip}{0pt}
\item $k=2$ and $\tilde D_2 = P_3$.
\item $k=2$ and $\tilde D_2 = P_4$.
\item $k=3$ and $\tilde D_2 = \tilde D_3 = P_3$.
\end{enumerate}

We summarize the above discussion as follows:

Let $G$ be a graph in $\mathcal{O}_2$. Each $2$-cut $\{a,b\}$ of $G$ has a unique component $B_{ab}$ of $G \setminus \{a,b\}$ such that $\tilde B_{ab}$ is bad. We call $G \setminus B_{ab}$ the \emph{good-side} of $\{a,b\}$. 
The good-side of a $2$-cut $\{a,b\}$ is either
\begin{itemize}
\item a path of length $2$ between $a$ and $b$,
\item a path of length $3$ between $a$ and $b$, or
\item a $K_{2,2}$ where $a$ and $b$ are non-adjacent
\end{itemize}

We shall show below that every graph in $\mathcal{O}_2$ can be constructed from a small 3-connected graph by replacing some of its edges by some of the three graphs in Figure~\ref{fig:2connsubd}. To state precise, we call the replacement of an edge $ab$ with $P_3 = acb$, $P_4 = acdb$ and $K_{2,2} = acb \cup adb$, respectively, as \emph{$1$-subdivision, $2$-subdivision and $11$-subdivision} where $c,d$ are adjacent to no other vertices; see Figure~\ref{fig:2connsubd}. We call these three operations as \emph{good-subdivisions}.

\begin{figure}[h!]
\center
\includegraphics[scale=1]{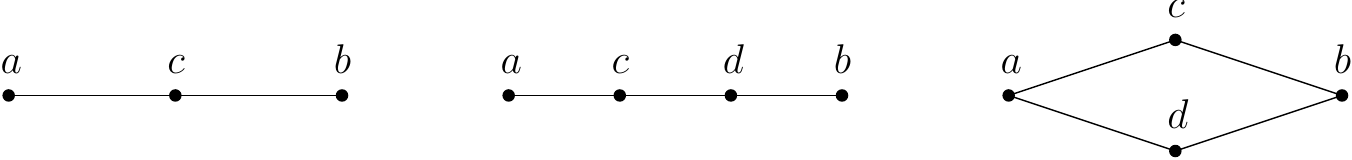}
\caption{Three ways of replacing an edge $ab$}\label{fig:2connsubd}
\end{figure}

\begin{lemma} \label{lem:3connbase}
Every graph in $\mathcal{O}_2$ is obtained from a $3$-connected graph on $4$,~$5$, or~$6$ vertices by good-subdividing some of its edges.
\end{lemma}

\begin{proof}
Let us consider the inclusion-wise maximal good-sides of $2$-cuts.
We would like to replace each of them with an edge between the vertices in its $2$-cut. 
To make this operation valid, 
we begin with showing that if two good-sides intersect, 
then both of them are contained in a good-side that is $P_4$, or 
the intersection is a single vertex contained in both of their $2$-cuts. 
Note that if $\tilde A$ is bad then $A$ has at least $5$ vertices, as $K_6$ has a good tree-representation for every pair of its vertices.

Let $G$ be a graph in $\mathcal{O}_2$.
Suppose that $G$ has two $2$-cuts $\{a,b\}$ and $\{c,d\}$ such that $c$ is in $G \setminus \tilde B_{ab}$. If $d$ is in $B_{ab}$, then $d$ must be a cut-vertex of $\tilde B_{ab}$ separating $a$ from $b$. The subgraph $\tilde B_{ab}$ has precisely two blocks, namely $D_a$ and $D_b$, and we assume that $a \in D_a$ and $b \in D_b$. 
By Lemmas 3.4, 3.5 and 3.7, $ab \notin E(G)$ and $cd \notin E(G)$. That is, both $\{a,d\}$ and $\{b,d\}$ are 2-cuts of $G$. Since $D_a \cup D_b$ is bad, $G \setminus (D_a \cup D_b)$ has at most two vertices. Considering the bad-sides of $\{a,d\}$ and $\{b,d\}$, we deduce that precisely one of $D_a$ and $D_b$, let us say $D_a$, is bad. Then the good-side of $\{a,d\}$ already has $\{a,b,c,d\}$ so that it must be the path $acdb$, which contains the good-sides of $\{a,b\}$ and $\{c,d\}$.

Therefore, if we consider only the inclusion-wise maximal good-sides, then their pairwise intersections have size at most $1$ and we can safely replace all of them at once by edges. Let $H$ be the resulting proper minor of $G$. If $H$ has a 2-cut, then we construct $G$ back from $H$ and the 2-cut still remains in $G$, which is impossible since 
for each $2$-cut $S$, we remove all but one component of $G\setminus S$ while producing $H$.
If $H$ has at least 7 vertices, then $H$ has a minor in $\mathcal{O}_3$ so that $G \notin \mathcal{O}_2$. Thus $H$ has at most 6 vertices. Obviously $H$ cannot be $K_2$. If $H$ is a triangle $abc$, then $G$ is obtained from $abc$ by good-subdividing all three edges $ab, bc$ and $ca$. To find a tree-representation of $G$ with width 2, we start from a $K_{1,3}$ where its three edges are labelled respectively by $ab, bc$ and $ca$. Then we can add the good-sides for the edges $ab, bc$ and $ca$ without increasing the width. Hence $H$ has $4$, $5$, or $6$ vertices and is $3$-connected.
\end{proof}

The obstructions obtained from a $3$-connected graph on $4$, $5$, and $6$ vertices respectively are listed in Figures~\ref{fig:o4},~\ref{fig:o5}, and~\ref{fig:o6}. The respective proofs are given in Lemmas~\ref{lem:o4},~\ref{lem:o5} and~\ref{lem:o6}. Note that Lemmas~\ref{lem:bad+p4} and~\ref{lem:bad+c4} imply the following.

\begin{lemma} \label{lem:p4c4equiv}
Let $G$ be a graph with an induced path $acdb$ such that $c$ and $d$ are non-adjacent to other vertices. Let $H$ be the graph obtained from $G-\{c,d\}$ by adding two new vertices $c'$,$d'$ and paths $ac'b$ and $ad'b$. Then $G \in \mathcal{O}_2$ if and only if $H \in \mathcal{O}_2$.
\end{lemma}

Hence in the following discussion we do not consider $11$-subdivisions. The obstructions obtained by replacing $2$-subdivisions with $11$-subdivisions shall be added to the list without mentioning.

We shall use the following lemma often when we show a graph has \MMW at most $2$.
\begin{lemma}\label{lem:mmw2}
Let $G$ be a graph. 
If $\{V_1, V_2, V_3\}$ is a partition of $V(G)$ and $G$ has six vertices $a_i, b_i$ for $i=1,2,3$ such that for each $i$,
\begin{itemize}
\item[(1)] $a_i, b_i \in V_i$,
\item[(2)] $\{a_i, b_i\}$ separates $V_i$ from $V(G) \setminus V_i$, and
\item[(3)] $(G[V_i], \{a_i, b_i\})$ has a good tree-representation,
\end{itemize}
then $\mmw(G) \leq 2$.
\end{lemma}
\begin{proof}
For each $i$, we consider a good tree-representation of $G[V_i]$ such that the subtrees for $a_i$ and $b_i$ share an edge whose one end has degree $1$. We combine the three tree-representations by identifying those degree 1 vertices to obtain a tree-representation of $G$ with width at most $2$.
\end{proof}

The way we use Lemma 3.10 to show a graph has mm-width $\leq 2$ is that, we try to cover the graph with either three good-sides or two good-sides and a set of at most two vertices. If we do so, the sets become $V_1, V_2$ and $V_3$ in the statement and Lemma 3.10 applies.

For convenience, we state here at once that the graphs in the following Lemmas \ref{lem:o4}, \ref{lem:o5}, \ref{lem:o6} all have \MMW at least $3$. Lemma~\ref{lem:tangle3} is a corollary of the following lemma.

\begin{lemma} \label{lem:mmw3}
Let $\mathcal{O}_4$, $\mathcal{O}_5$ and $\mathcal{O}_6$, respectively, be the set of graphs in Figures~\ref{fig:o4},~\ref{fig:o5} and~\ref{fig:o6}. Every graph in $\mathcal{O}_4 \cup \mathcal{O}_5 \cup \mathcal{O}_6$ has maximum matching width at least $3$.
\end{lemma}
\begin{proof}
By Theorem~\ref{thm:tangle}, it is enough to give a tangle of order $3$.
We shall explain how to find a tangle of order $3$ for each of those graphs. Recall that a tangle of order $3$ contains all the `smaller' sets $X$ with $\mm (X) \leq 2$.

Let $G$ be a graph and let $X$ be a subset of $V(G)$ such that $\mm (X) \leq 2$ and $\abs{X}\ge 3$. In other words, the bipartite graph on $V(G)$ with all the edges in $E(G)$ having one end in $X$ and the other not in $X$ has maximum matching size $2$. 
Thus we can find a set $\{a,b\}$ that is a $2$-cut of $G$ and $G \setminus \{a,b\}$ has a component, say $S$, such that $\tilde S$ contains either $X$ or $V(G) \setminus X$.

Therefore, for each graph $G \in \mathcal{O}_4 \cup \mathcal{O}_5 \cup \mathcal{O}_6$, we set $S_G$ to be the collection of all vertex subsets of the following three types:
\begin{itemize}
\setlength{\itemsep}{0pt}
\setlength{\parsep}{0pt}
\setlength{\parskip}{0pt}
\item a set of size at most $2$
\item a good-side of a $2$-cut
\item if $G$ has a 11-subdivision made of the paths $aub$ and $avb$, then $S_G$ contains both $\{a,u,b\}$ and $\{a,v,b\}$.
\end{itemize}
Now we consider the tangle axioms (T1), (T2) and (T3) in Section \ref{sec:tangle} to verify that $S_G$ is a tangle. (T1) follows immediately from the above discussion, and (T3) is also each to check for all graphs in $G \in \mathcal{O}_4 \cup \mathcal{O}_5 \cup \mathcal{O}_6$. For (T2), we can check that no three good-sides cover the whole graph and it remains to see that there are no two good-sides that covers all but at most two vertices. We leave the detail to the reader.
\end{proof}

Now we consider the case when the 3-connected graph in Lemma \ref{lem:3connbase} has four vertices. The only 3-connected graph on four vertices is $K_4$.

\begin{figure}[h]
        \centering
        \begin{subfigure}[b]{0.3\textwidth}
            \centering
            \includegraphics[height=50pt]{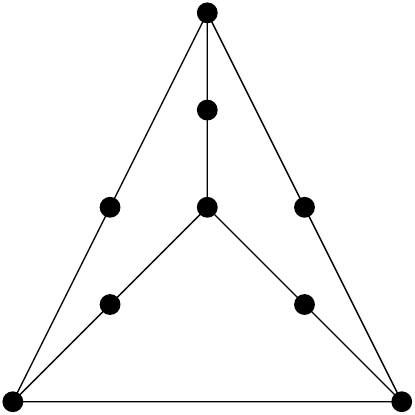}
            \caption{$\mathcal{O}_4^1$}\label{fig:o41}
        \end{subfigure}
        \begin{subfigure}[b]{0.3\textwidth}
            \centering
            \includegraphics[height=50pt]{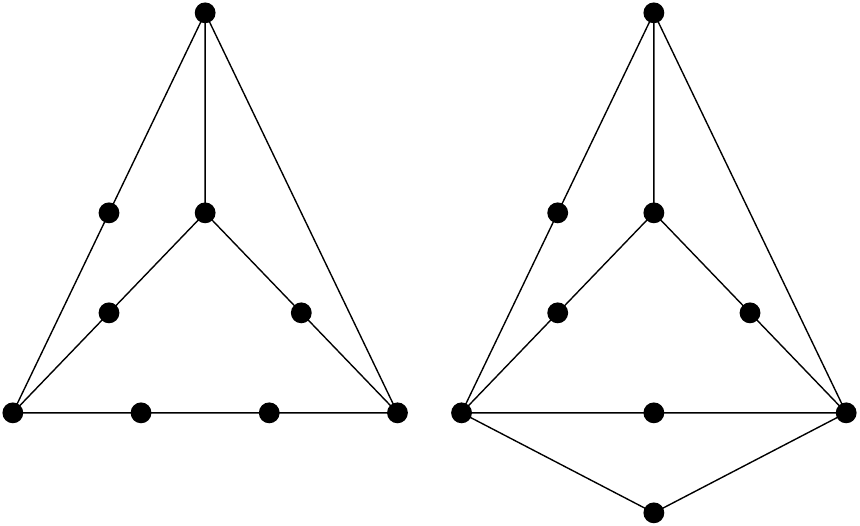}
            \caption{$\mathcal{O}_4^2$}\label{fig:o42}
        \end{subfigure}

        \begin{subfigure}[b]{\textwidth}
            \centering
            \includegraphics[height=50pt]{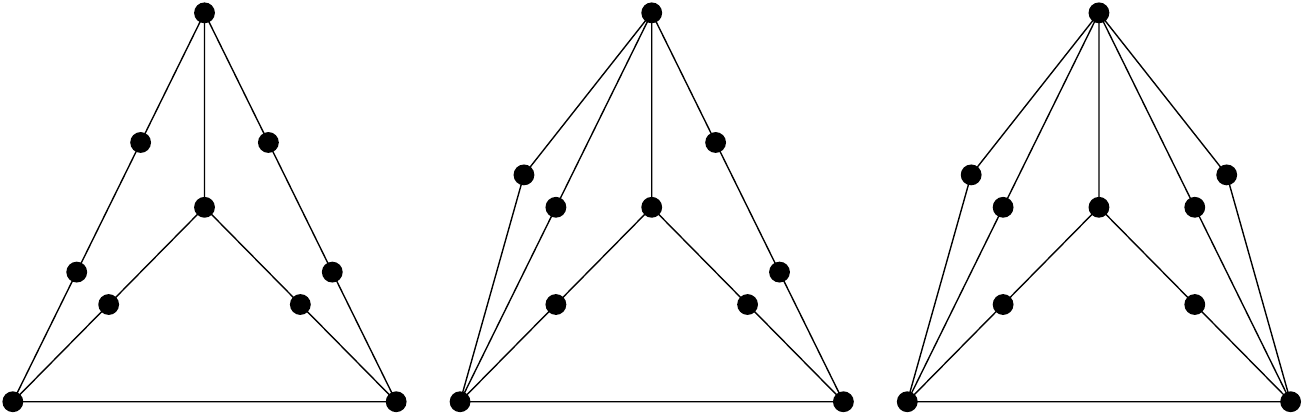}
            \caption{$\mathcal{O}_4^3$}\label{fig:o43}
        \end{subfigure}

	\begin{subfigure}[b]{\textwidth}
            \centering
            \includegraphics[height=50pt]{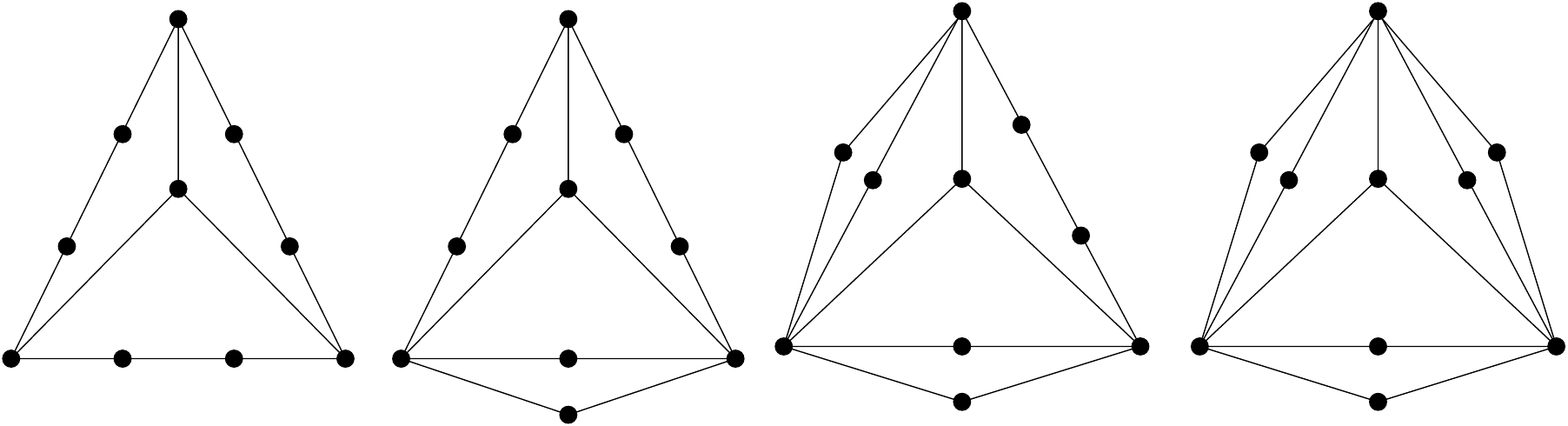}
            \caption{$\mathcal{O}_4^4$}\label{fig:o44}
        \end{subfigure}
        \caption{The graphs in $\mathcal{O}_4=\mathcal{O}_4^1\cup\mathcal{O}_4^2\cup\mathcal{O}_4^3\cup\mathcal{O}_4^4$}\label{fig:o4}
\end{figure}

\begin{lemma}\label{lem:o4}
Let $\mathcal{O}_4^1, \mathcal{O}_4^2, \mathcal{O}_4^3, \mathcal{O}_4^4$ be the sets of graphs in Figure~\ref{fig:o4}.
If a graph $G$ is obtained from $K_4$ by good-subdividing some of its edges, 
then $G\in \mathcal{O}_2$ if and only if $G\in \mathcal{O}_4=\mathcal{O}_4^1 \cup \mathcal{O}_4^2\cup \mathcal{O}_4^3\cup \mathcal{O}_4^4$. 
\end{lemma}
\begin{proof}
By Lemma \ref{lem:mmw3} the graphs in $\mathcal{O}_4$ has \MMW at least $3$.
It can be easily checked that all their proper minors have \MMW at most $2$ using Lemma~\ref{lem:mmw2}.

Now we consider the graphs obtained from $K_4$ by good-subdivisions. 
We divide the cases via the number of good-subdivisions.
Recall that by Lemma~\ref{lem:p4c4equiv}, we only consider $2$-subdivision and not $11$-subdivision.

If $G$ has no $2$-subdivision and has at most four $1$-subdivisions, then $\mmw(G)\le 2$ by Lemma~\ref{lem:mmw2}. 
The unique graph with no $2$-subdivision and five $1$-subdivisions is in $\mathcal{O}_4^1$.

%

If $G$ has one $2$-subdivision and at most three $1$-subdivisions, 
then $\mmw(G)\le 2$ by Lemma~\ref{lem:mmw2} unless $G$ is the first graph in $\mathcal{O}_4^2$.
If $G$ has one $2$-subdivision and four $1$-subdivisions, then $G$ contains the graph in $\mathcal{O}_4^1$ as a minor.

If $G$ has two $2$-subdivisions and at most two $1$-subdivisions,
then either $G$ has \MMW $2$,
$G$ contains a graph in $\mathcal{O}_4^2$ as a minor, or
$G$ is the first graph in $\mathcal{O}_4^3$. The rest of $\mathcal{O}_4^3$ is obtained by replacing 2-subdivisions with 11-subdivisions; see Lemma \ref{lem:p4c4equiv}.
If $G$ has more than two 1-subdivisions, then $G$ contains the graph in $\mathcal{O}_4^1$ as a minor.

If $G$ has three 2-subdivisions and no $1$-subdivision, then $\mmw(G)\le 2$ by Lemma~\ref{lem:mmw2} unless $G$ is the first graph in $\mathcal{O}_4^4$.
If $G$ has three 2-subdivisions and at least one 1-subdivision, 
then $G$ contains a graph in $\mathcal{O}_4^2 \cup \mathcal{O}_4^3$ as a minor.
\end{proof}

Lemma~\ref{lem:o5} and~\ref{lem:o6}, respectively, characterizes the graphs in $O_2$ that is obtained from a $3$-connected graph onfive and six vertices.

\begin{figure}[h]
        \centering
        \begin{subfigure}[b]{\textwidth}
            \centering
            \includegraphics[height=30pt]{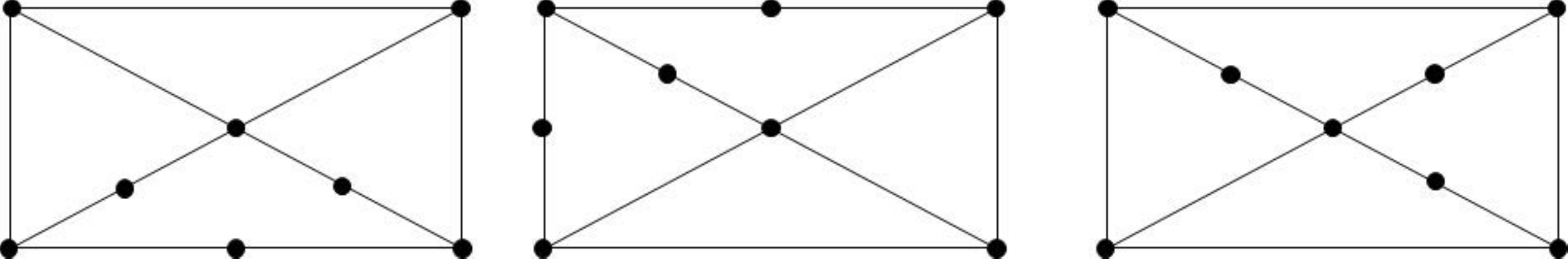}
            \caption{$\mathcal{O}_5^1$}\label{fig:o51}
        \end{subfigure}
        \begin{subfigure}[b]{\textwidth}
            \centering
            \includegraphics[height=60pt]{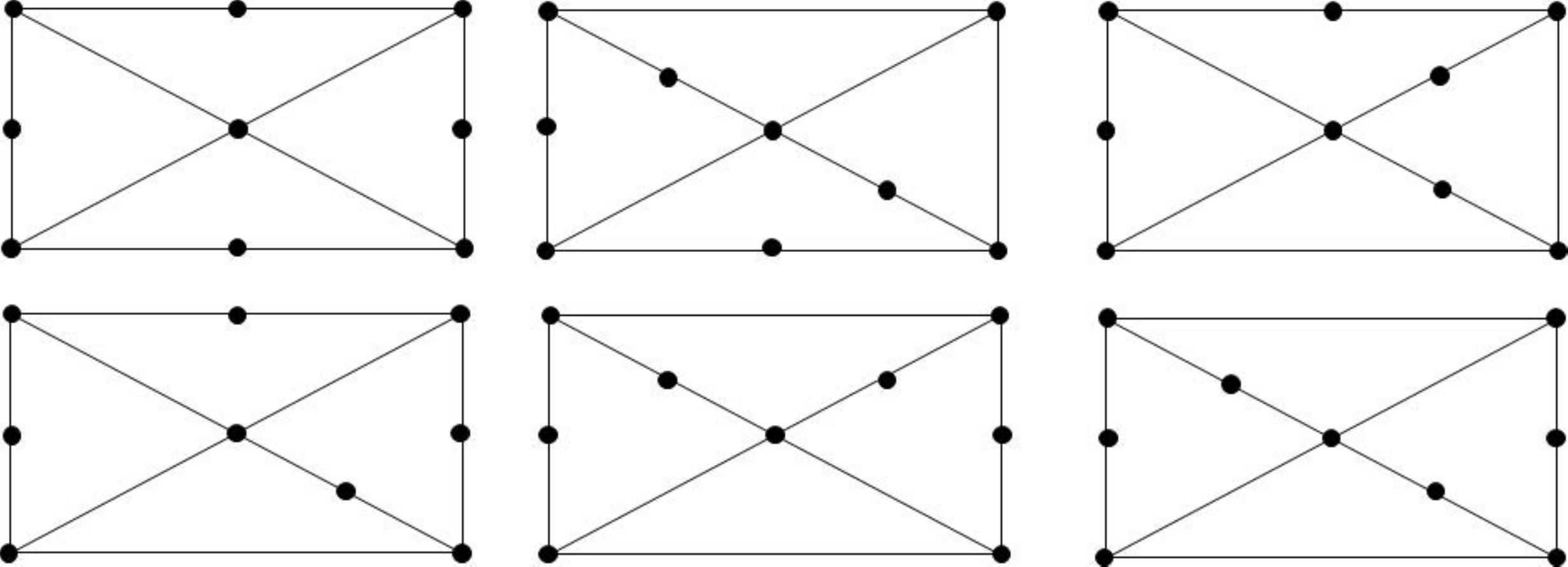}
            \caption{$\mathcal{O}_5^2$}\label{fig:o52}
        \end{subfigure}

	\begin{subfigure}[b]{0.3\textwidth}
            \centering
            \includegraphics[height=60pt]{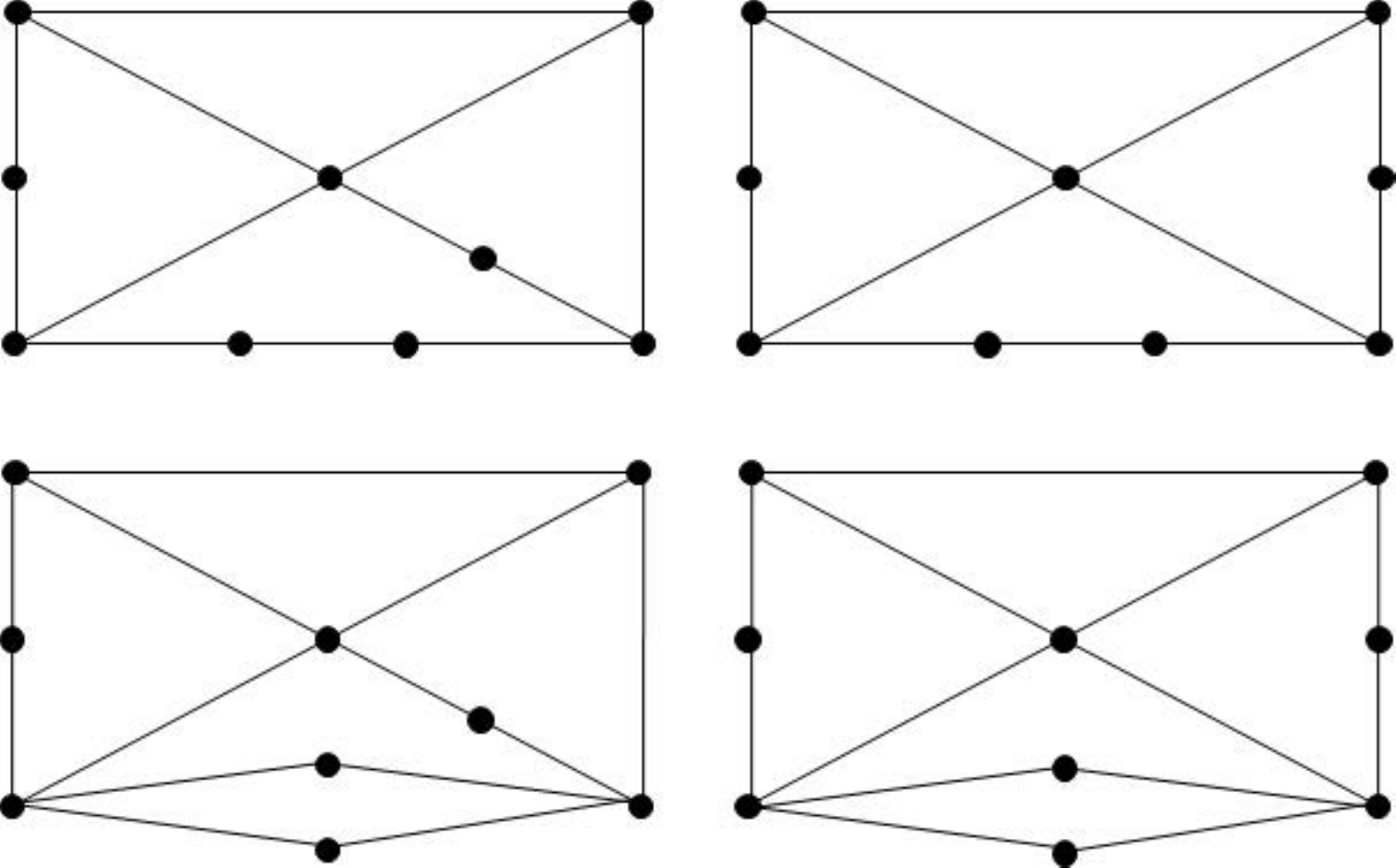}
            \caption{$\mathcal{O}_5^3$}\label{fig:o53}
        \end{subfigure}
	\begin{subfigure}[b]{0.3\textwidth}
            \centering
            \includegraphics[height=60pt]{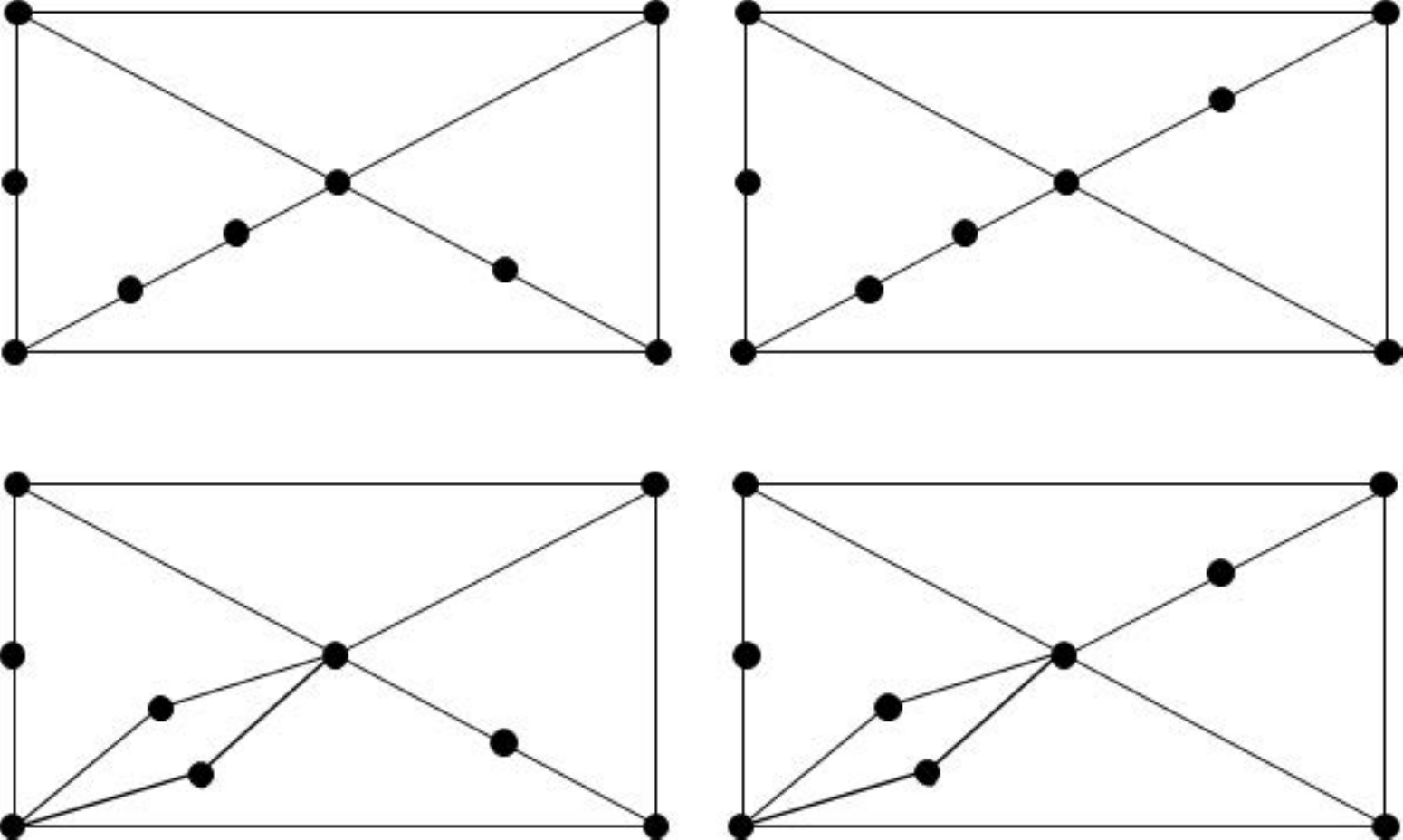}
            \caption{$\mathcal{O}_5^4$}\label{fig:o54}
        \end{subfigure}
	\begin{subfigure}[b]{0.3 \textwidth}
            \centering
            \includegraphics[height=60pt]{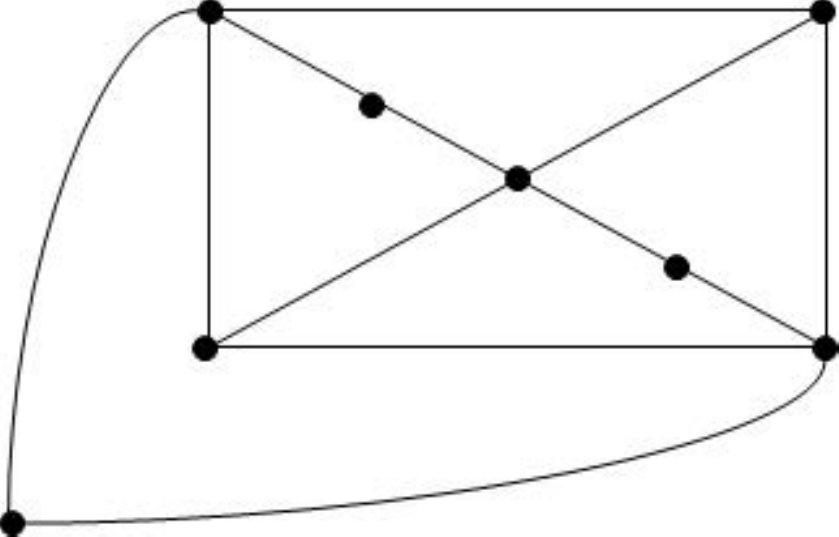}
            \caption{$\mathcal{O}_5^5$}\label{fig:o55}
        \end{subfigure}

        \caption{The graphs in $\mathcal{O}_5=\mathcal{O}_5^1\cup\mathcal{O}_5^2\cup\mathcal{O}_5^3\cup\mathcal{O}_5^4\cup\mathcal{O}_5^5$}\label{fig:o5}
\end{figure}

\begin{lemma}\label{lem:o5}
Let $\mathcal{O}_5^1,\mathcal{O}_5^2,\mathcal{O}_5^3,\mathcal{O}_5^4,\mathcal{O}_5^5$ be the sets of graphs in Figure~\ref{fig:o5}.
If a graph $G$ is obtained from a $3$-connected graph on $5$ vertices by good-subdividing some edges, 
then $G\in \mathcal{O}_2$ if and only if 
$G\in \mathcal{O}_5=\mathcal{O}_5^1\cup\mathcal{O}_5^2\cup\mathcal{O}_5^3\cup\mathcal{O}_5^4\cup\mathcal{O}_5^5$. 
\end{lemma}
\begin{proof}
By Lemma \ref{lem:mmw3} the graphs in $\mathcal{O}_5$ has \MMW at least 3. Their proper minors have \MMW $2$ by Lemma \ref{lem:mmw2} and hence they are in $\mathcal{O}_2$.

Now we consider the graphs that are also obtained from a 3-connected graph on 5 vertices by good-subdivisions. There are three 3-connected graphs on 5 vertices, namely the wheel $W_5$, $W_5$ plus an edge (say $W_5'$), and $K_5$.

Let us begin with $W_5$. Let $G$ be a graph obtained from $W_5$ by good-subdividing some edges.

Suppose that $G$ has no 2-subdivision and has three 1-subdivisions. If the to-be-subdivided edges of $W_5$ contain two independent edges, then Lemma~\ref{lem:mmw2} implies $\mmw(G) \leq 2$. Thus $G\in \mathcal{O}_2$ if and only if $G\in \mathcal{O}_5^1$. 
If $G$ has no 2-subdivision and has four 1-subdivisions, then it has \MMW $3$; a tangle of order 3 can be found in each case as in Lemma \ref{lem:mmw3}.
So in this case $G\in \mathcal{O}_2$ if and only if it does not have a graph in $\mathcal{O}_5^1$ as a minor. 
These are the graphs in $\mathcal{O}_5^2$.

If the number of $1$-subdivisions and $2$-subdivisions in $G$ is at least $4$, then $G$ contains a graph in $\mathcal{O}_5^1\cup\mathcal{O}_5^2$ as a minor. Suppose $G$ has one $2$-subdivision and two $1$-subdivisions. If the good-side of the 2-subdivision does not intersect with one of the other two good-sides, then Lemma~\ref{lem:mmw2} implies $\mmw (G) \leq 2$. Thus both good-sides of the $1$-subdivisions intersect with the good-side of the $2$-subdivision. If the $2$-subdivision happens at an edge incident with the vertex of degree $4$ in $W_5$, then $G \in \mathcal{O}_2$ if and only if $G$ is one of the top two graphs in $\mathcal{O}_5^4$; other cases contain a graph in $\mathcal{O}_5^1$ as a minor. The bottom two graphs in $\mathcal{O}_5^4$ are obtained by replacing the $2$-subdivision with a $11$-subdivision. If the $2$-subdivision is not incident with the degree-$4$ vertex of $W_5$, then we get the graphs in $\mathcal{O}_5^3$.

If $G$ has at least two $2$-subdivisions, then either $\mmw(G) \leq 2$
or it contains a graph in $\mathcal{O}_5^1\cup\mathcal{O}_5^2\cup\mathcal{O}_5^3\cup\mathcal{O}_5^4$ as a minor. It completes the graphs obtained from $W_5$.

Now we consider the graphs $G \in \mathcal{O}_2$ obtained from $W_5'$ by good-subdivisions. The graph $W_5'$ has three edges whose removal results in $W_5$. Suppose one of these three edges, say $e$, is not good-subdivided in $G$. If $G$ has at least four good-subdivisions, then $G-e$ contains a graph in $\mathcal{O}_5^1\cup\mathcal{O}_5^2\cup\mathcal{O}_5^3\cup\mathcal{O}_5^4$ as a minor. If $G$ has at most three good-subdivisions and $G-e$ does not contain a graph in $\mathcal{O}_5^1$ as a minor, then  Lemma \ref{lem:mmw2} implies $\mmw(G) \leq 2$. Hence all three edges of $W_5'$ in the triangle of degree-4 vertices must be good-subdivided in $G$. Since the graph in $\mathcal{O}_5^5$ is in $\mathcal{O}_2$, it is the unique graph obtained from $W_5'$ in $\mathcal{O}_2$.

The last $3$-connected graph on five vertices is $K_5$. Let $G$ be a graph obtained from $K_5$ by good-subdivisions. Using an argument similar to above we can show that every edge of $K_5$ must be subdivided. Hence $G \notin \mathcal{O}_2$ and $\mathcal{O}_5$ is the precise set of obstructions obtained from a $3$-connected graph on five vertices.
\end{proof}

\begin{figure}
        \centering
        \begin{subfigure}[b]{0.2\textwidth}
            \centering
            \includegraphics[height=50pt]{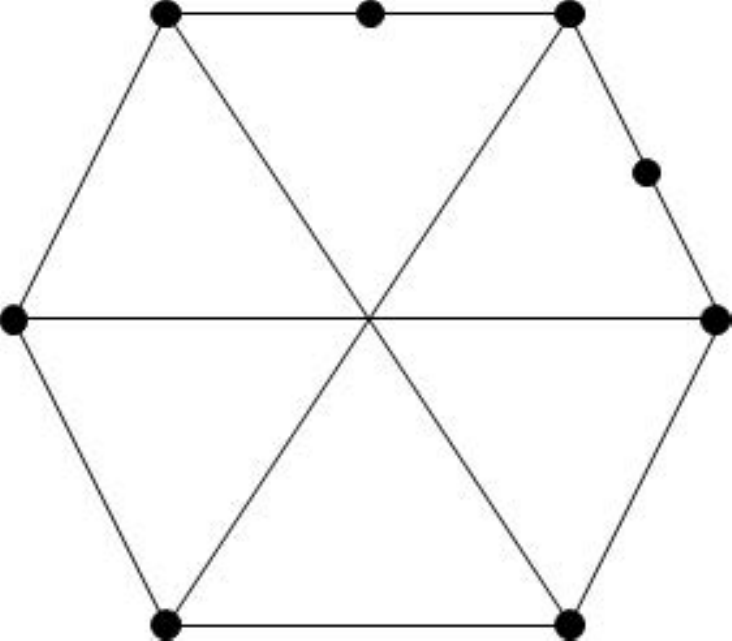}
            \caption{$\mathcal{O}_6^1$}\label{fig:o61}
        \end{subfigure}
        \begin{subfigure}[b]{0.2\textwidth}
            \centering
            \includegraphics[height=50pt]{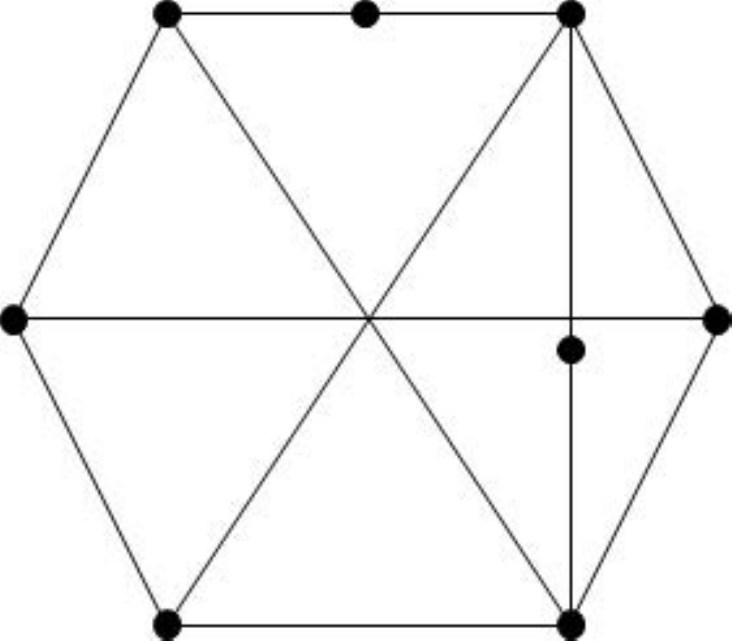}
            \caption{$\mathcal{O}_6^2$}\label{fig:o62}
        \end{subfigure}
        \begin{subfigure}[b]{0.2\textwidth}
            \centering
            \includegraphics[height=50pt]{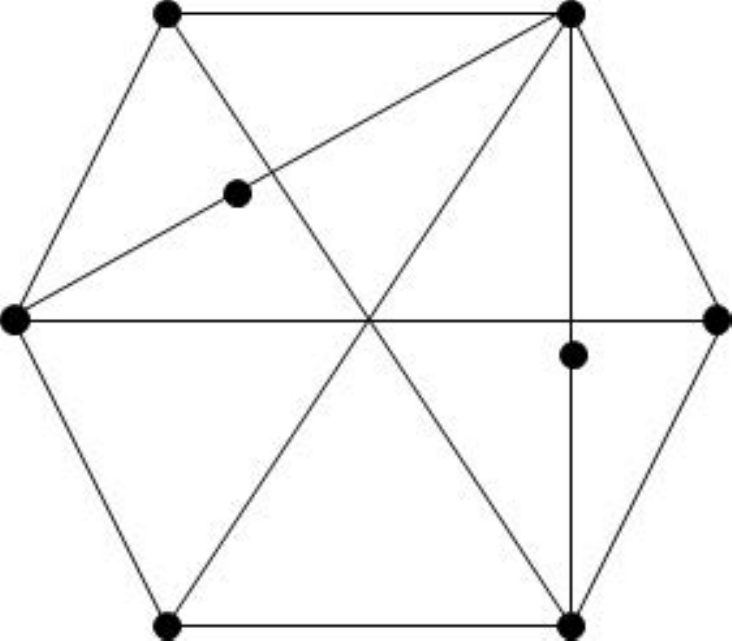}
            \caption{$\mathcal{O}_6^3$}\label{fig:o63}
        \end{subfigure}
        \vspace{0.7cm}
        
	   \begin{subfigure}[b]{0.35\textwidth}
            \centering
            \includegraphics[height=40pt, width=150pt]{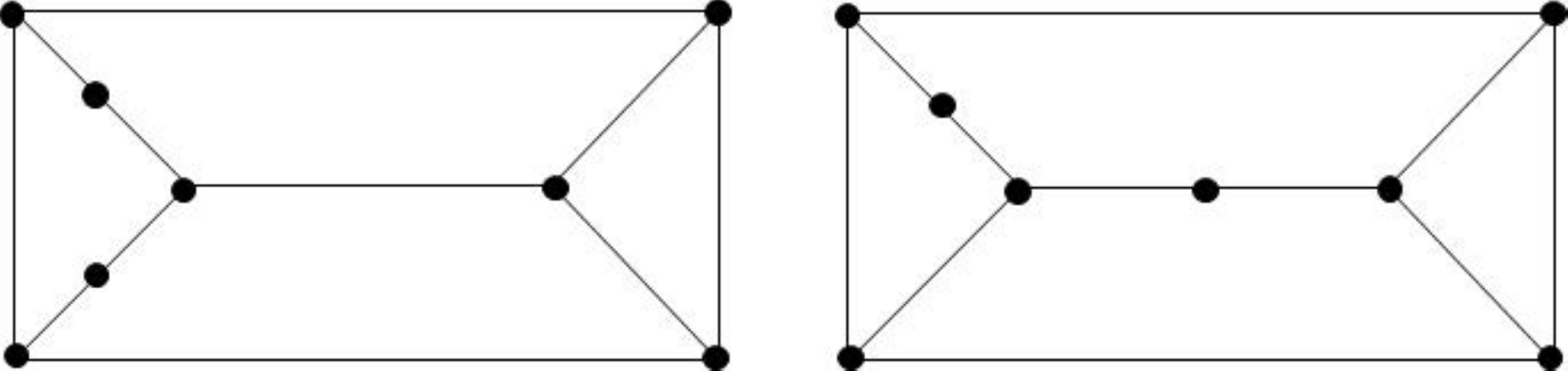}
            \caption{$\mathcal{O}_6^4$}\label{fig:o64}
        \end{subfigure}~~~~~~~~~~~~
        \begin{subfigure}[b]{0.45\textwidth}
            \centering
            \includegraphics[height=40pt, width=220pt]{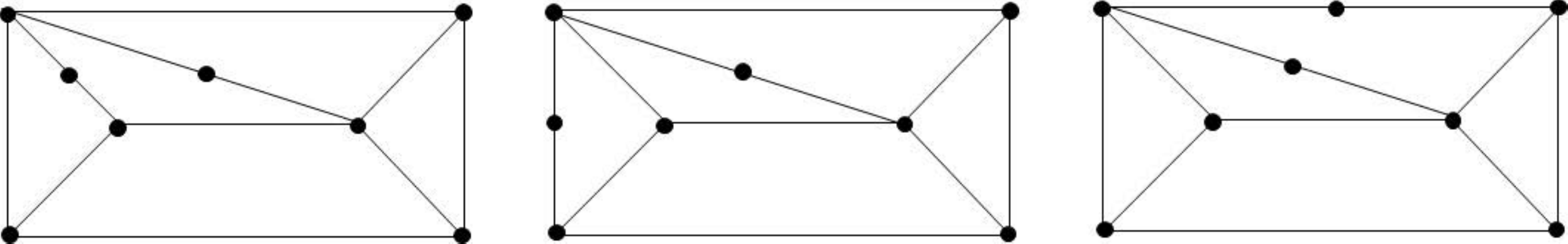}
            \caption{$\mathcal{O}_6^5$}\label{fig:o65}
        \end{subfigure}
       \vspace{0.7cm}
 
        \begin{subfigure}[b]{\textwidth}
            \centering
            \includegraphics[height=60pt]{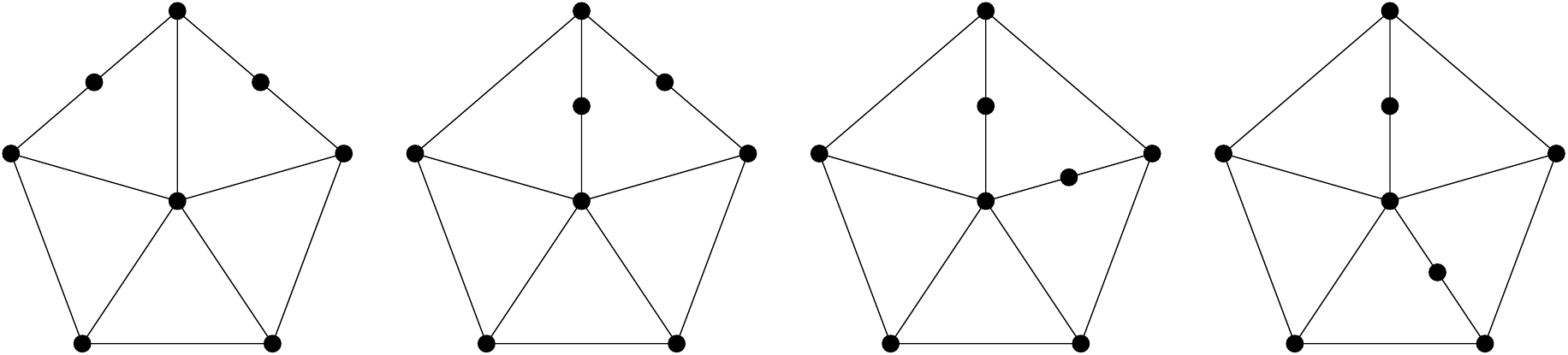}
            \caption{$\mathcal{O}_6^6$}\label{fig:o66}
        \end{subfigure}

        \caption{The graphs in $\mathcal{O}_6=\mathcal{O}_6^1\cup\mathcal{O}_6^2\cup\mathcal{O}_6^3\cup\mathcal{O}_6^4\cup\mathcal{O}_6^5\cup\mathcal{O}_6^6$}\label{fig:o6}
\end{figure}

\begin{lemma}\label{lem:o6}
Let $\mathcal{O}_6^1, \mathcal{O}_6^2, \mathcal{O}_6^3, \mathcal{O}_6^4, \mathcal{O}_6^5, \mathcal{O}_6^6$ be the sets of graphs in Figure~\ref{fig:o6}.
If a graph $G$ is obtained from a $3$-connected graph on $6$ vertices by good-subdivisions,
then $G\in \mathcal{O}_2$ if and only if $G\in \mathcal{O}_6=\mathcal{O}_6^1 \cup \mathcal{O}_6^2\cup \mathcal{O}_6^3\cup \mathcal{O}_6^4\cup \mathcal{O}_6^5\cup \mathcal{O}_6^6$.
\end{lemma}
\begin{proof}
Let $H$ be a 3-connected graph on six vertices and let $G$ be a graph obtained from $H$ by good-subdividing some edges. If two adjacent edges of $H$ are good-subdivided in $G$, then we can find a tangle of order 3 in $G$ and hence $\mmw(G) \geq 3$; all graphs in $\mathcal{O}_6$ are of this type. If there is no such pair in $H$, then the good-subdivisions happened at a matching of $H$ and Lemma~\ref{lem:mmw2} implies $\mmw(G) \leq 2$. We leave it to the reader to check that the proper minors of the graphs in $\mathcal{O}_6$ have \MMW at most 2.

If $H$ is minimally 3-connected, then all the graphs obtainable from $H$ by good-subdividing two adjacent edges are in $\mathcal{O}_6$; $\mathcal{O}_6^1$ for $K_{3,3}$, $\mathcal{O}_6^4$ for the prism and $\mathcal{O}_6^6$ for the wheel $W_6$.

Suppose that $H$ is not minimally 3-connected and $G \in \mathcal{O}_2$. Let $e$ be an edge of $H$ such that $H-e$ is still 3-connected. If $e$ is not subdivided in $G$, then by the above discussion $G-e$ has two adjacent good-sides and $\mmw(G-e) \geq 3$, a contradiction. Thus $e$ must be good-subdivided in $G$ and $H$ has at most two edges whose removal does not affect its 3-connectivity. Note that if $H$ has two such edges, then they should be also adjacent.

If $H$ is $K_{3,3}$ plus an edge, then the additional edge must be subdivided and we need another adjacent edge to subdivide. But independently of this choice the resulting graph is isomorphic to the graph in $\mathcal{O}_6^2$. There is a unique way of adding two adjacent edges to $K_{3,3}$ and the graphs in $\mathcal{O}_6^3$ is the result of subdividing both.

If $H$ is the prism plus an edge, then we have three non-isomorphic choices of another adjacent edge to subdivide. They are in $\mathcal{O}_6^5$. There is again a unique way of adding two adjacent edges to the prism but it contains a graph in $\mathcal{O}_6^6$ as a minor.

There is a unique (up to isomorphism) way to add an edge to $W_6$ but it already has three edges that are removable while maintaining 3-connecitivity. Thus the list is complete.
\end{proof}

By Lemma~\ref{lem:O3}, a graph is in the obstruction set and $3$-connected if and only if it is in $\mathcal{O}_3$.
If $G$ is in the obstruction set but not $3$-connected, 
then it should be obtained from a $3$-connected graph on $4$, $5$, or $6$ vertices by Lemma~\ref{lem:3connbase}.
Lemmas~\ref{lem:o4},~\ref{lem:o5}, and~\ref{lem:o6} show that $G\in \mathcal{O}_4\cup\mathcal{O}_5\cup \mathcal{O}_6$. 
Therefore, the following theorem holds: 

\begin{theorem}\label{thm:2}
Let $\mathcal{O}=\mathcal{O}_3\cup\mathcal{O}_4\cup\mathcal{O}_5\cup\mathcal{O}_6$ be the set of $45$ graphs in Figures~\ref{fig:3conn},\ref{fig:o4},\ref{fig:o5},\ref{fig:o6}. A graph $G$ has maximum matching width at most 2 if and only if $G$ has no minor isomorphic to a graph in $\mathcal{O}$.
\end{theorem}

\section{$k \times k$-grid}\label{sec:grid}


The \emph{$k\times k$-grid}, denoted by $G_k$, is the graph with a vertex set $V(G_k)=\{(i,j):1\le i,j \le k\}$ and an edge set $E(G_k)=\{(i,j)(i',j'):\abs{i-i'}+\abs{j-j'}=1 \}$. In this section, we show $\mmw(G_k) = k$ for $k \geq 2$. 

Vatshelle~\cite{Vatshelle2012} showed the following inequality. Recall that $\rw(G)$ and $\brw(G)$ respectively denotes the rank-width and the branch-width of $G$.
\begin{theorem}[\cite{Vatshelle2012}]\label{thm:vatshelle}
If $G$ is a graph, then 
\[
\rw(G)\le\mmw(G)\le\max(\brw(G),1).
\]
\end{theorem}
It is known that $\brw(G_k)=k$~\cite{RS1991} and $\rw(G_k)=k-1$~\cite{Jelinek2010}.
Hence $\mmw(G_k)$ is either $k-1$ or $k$.
We shall show $\mmw(G_k) > k-1$ by finding a \emph{tangle} of order $k$; see Section \ref{sec:tangle}. We assume $k \geq 2$ throughout this section.

Let $C_i = \{(i,j):1\le j\le k\}$ and $R_j =\{(i,j):1\le i\le k\}$ be the set of vertices on the $i$-th column and the $j$-th row respectively. Recall that for a vertex set $X \subseteq V(G)$, $\mm_G(X)$ denotes the size of a maximum matching in $G[X,V(G) \setminus X]$. 
We omit $G_k$ in $\mm_{G_k}$ and write $\mm(X) = \mm_{G_k}(X)$ in this section. Let $X^c = V(G_k) \setminus X$ for $X \subseteq V(G_k)$.

\begin{lemma}\label{rc}
If $X\subseteq V(G_k)$ and $\mm(X)< k$, then $R_i\subseteq X$ for some $i$ if and only if $C_j\subseteq X$ for some $j$.
\end{lemma}
\begin{proof}
Suppose that $R_i\subseteq X$ for some $i$. Then each $C_j$ intersects with $X$.
If $C_j\nsubseteq X$ for every $j$, each $G[C_j]$ contains an edge with one end in $X$ and the other end in $X^c$. 
Since these edges form a matching of size $k$, we have $\mm(X)\ge k$ which is a contradiction. 
Thus $C_j\subseteq X$ for some $j$.
The converse follows from the symmetry.
\end{proof}

For $X\subseteq V(G_k)$, we say that $X$ is \emph{small} if $\mm(X)<k$ and $R_i \not\subseteq X$ for all $i=1,2,\ldots,k$. Note that, by Lemma \ref{rc}, $C_j \not\subseteq X$ for all $j=1,2,\ldots,k$ if $X$ is small.

\begin{lemma}\label{t1}
Let $X\subseteq V(G_k)$. If $\mm(X)<k$, then one of $X$ and $X^c$ is small.
\end{lemma}
\begin{proof}
Suppose that neither $X$ nor $X^c$ is small. Then we can choose $i_1$, $i_2$ with $1\le i_1,i_2\le k$ such that $R_{i_1}\subseteq X$ and $R_{i_2}\subseteq X^c$. Now we may choose an edge from each column of $G_k$ with endpoints one in $X$ and the other in $X^c$. Since these edges form a matching of size $k$, we have $\mm(X)\ge k$, a contradiction.
\end{proof}

\begin{lemma}\label{norc}
If $X\subseteq V(G_k)$ is small, then there exist $i,j$ such that $R_i \cap X = C_j \cap X = \emptyset$.
\end{lemma}
\begin{proof}
Suppose that $|R_i\cap X |>0$ for all $i$. Since $X$ is small, $R_i \cap X^c \not = \emptyset$. 
Thus, $G[R_i]$ contains an edge between $X$ and $X^c$ for every $i$. 
These edges show that $\mm(X)\ge k$, a contradiction. Likewise, $C_j \cap X = \emptyset$ for some $j$.
\end{proof}


\begin{lemma}\label{main}
If $X_1\cup X_2\cup X_3=V(G_k)$, 
then one of $X_1$, $X_2$, and $X_3$ is not small.
\end{lemma}
\begin{proof}
We prove by induction on $k$. The lemma is trivial when $k=2$. Assume that $k>2$ and the lemma is true for $k-1$. To prove by contradiction, let us suppose that all of $X_1$, $X_2$, and $X_3$ are small. Note that each row or column intersects at least two of $X_1$, $X_2$ and $X_3$.

Firstly we suppose that $R_k \cup C_k$ intersects $X_t$ for all $t \in \{1,2,3\}$. We consider the $(k-1) \times (k-1)$-grid $G_{k-1} = G_k \setminus (R_k \cup C_k)$ with sets $X_t' = X_t \setminus (R_k \cup C_k)$ for each $t \in \{1,2,3\}$ so that $X_1' \cup X_2' \cup X_3' = V(G_{k-1})$. By the induction hypothesis, we may assume that $X_1'$ is not small in $G_{k-1}$.
That is, $\mm_{G_{k-1}}(X_1') \geq k-1$ or $X_1'$ contains a row of $G_{k-1}$.
If $\mm_{G_{k-1}}(X_1') \geq k-1$, then $G_{k-1}$ has a matching of size $k-1$ between $X_1'$ and $V(G_{k-1}) \setminus X_1'$. Since $G_k$ has an edge in $G_k[R_k \cup C_k]$ with one end in $X_1$ and the other in $X_1^c$, we obtain a matching of size $k$ in $G_k[X_1, X_1^c]$ showing that $\mm(X_1) \geq k$ and $X_1$ is not small. Hence we may assume that $\mm_{G_{k-1}}(X_1') < k-1$ and $X_1'$ contains a row $R'$ of $G_{k-1}$. Since we assumed $X_1$ to be small, one of the columns of $G_k$ does not intersect $X_1$ by Lemma \ref{norc} but it must be $C_k$; all other columns intersect with $R'$. On the other hand, by Lemma \ref{rc}, $X_1'$ also contains a column of $G_{k-1}$ and $R_k$ does not intersect $X_1$. Thus $(R_k \cup C_k) \cap X_1 = \emptyset$, a contradiction to our assumption that $R_k \cup C_k$ intersects all of $X_1, X_2$ and $X_3$.

Therefore we may assume that for every choice $i,j \in \{1,k\}$, $R_i \cup C_j$ does not intersect all $X_t$ at the same time. Since each row and column intersects at least two of $X_1, X_2$ and $X_3$, if $R_1\cup R_k$ meets all $X_t$, then either $R_1\cup C_k$ or $R_k\cup C_k$ meets all $X_t$ so that we assume both $R_1$ and $R_k$ intersects $X_1$ and $X_2$ but not $X_3$. It follows also that both $C_1$ and $C_k$ intersects $X_1$ and $X_2$ but not $X_3$.

We shall show $\mm(X_1) + \mm(X_2) \geq 2k$ by proving that each column of $G_k$ contains either two independent edges from one of $E[X_1, X_1^c]$ and $E[X_2, X_2^c]$, or one edge from each set. Those edges form two matchings in $G[X_1, X_1^c]$ and $G[X_2, X_2^c]$ respectively whose sizes sum up to at least $2k$. Thus we get $\mm(X_1) \geq k$ or $\mm(X_2) \geq k$ and one of $X_1$ and $X_2$ is not small.

If a column has an edge with one end in $X_1\setminus X_2$ and the other in $X_2\setminus X_1$ then we are done. Thus $C_1$ and $C_k$ are fine. 
If all columns are as such then we are done. Otherwise, there is a column $C_i$ such that $C_i \cap X_2 \subseteq C_i \cap X_1$. Since $C_i \not\subset X_1$, we have $|C_i \cap (X_3 \setminus X_1)| > 0$. If $|C_i \cap (X_3 \setminus X_1)| \geq 2$ then $C_i$ has two independent edges in $E[X_1, X_1^c]$. Thus we assume $|C_i \cap (X_3 \setminus X_1)| = 1$, that is, $|C_i \cap X_1| = k-1$.
By Lemma~\ref{norc} we choose a column $C_j$ not intersecting with $X_1$, and between $C_i$ and $C_j$ we can find $k-1$ independent row-edges in $E[X_1, X_1^c]$. Since $C_1$ and $C_k$ are not in this area, we may choose an edge from $G[C_1] \cap G[X_1, X_1^c]$ and $G[X_1, X_1^c]$ has a matching of size $k$, showing that $\mm(X_1) \geq k$ and $X_1$ is not small. This final contradiction completes the proof.
\end{proof}


\begin{lemma}\label{lem:gridtangle}
Let $\mathcal{T}$ be the set of all small subsets of $V(G_k)$.
The set $\mathcal{T}$ is a tangle in $G_k$ of order $k$.
\end{lemma}
\begin{proof}
The first and the second axioms follow from Lemma~\ref{t1} and Lemma~\ref{main} respectively. For each $x \in V(G_k)$, the set $V(G_k)\setminus\{x\}$ contains a row and thus not in $\mathcal{T}$.
\end{proof}

By Theorem~\ref{thm:tangle}, Lemma~\ref{lem:gridtangle} implies $\mmw(G_k)> k-1$.
Since the branchwidth of $G_k$ is $k$, by Theorem~\ref{thm:vatshelle}, 
$\mmw(G_k)$ is at most $k$.

\begin{theorem}\label{thm:gridthm}
The $k\times k$-grid has maximum matching width $k$ for $k\ge 2$.
\end{theorem}

\section{Acknowledgments}
We thank Jan Arne Telle for introducing the first problem. 
We also thank Sang-il Oum for giving an important idea to compute the maximum matching width of the $k \times k$-grid.

\bibliographystyle{abbrv}
\bibliography{bib-width}

\end{document}